\documentclass[letterpaper,11pt]{article}
\usepackage{graphicx} 
\usepackage[]{amsmath,amssymb,amsfonts,latexsym,amsthm,enumerate,fullpage,xcolor,bbm,mathtools}
\usepackage[pagebackref,colorlinks,citecolor=blue,urlcolor=blue,linkcolor=black,,bookmarks=true]{hyperref}
\usepackage[nameinlink, noabbrev, capitalize]{cleveref}
\usepackage{amsmath, thm-restate}
\usepackage{pgfplots}
\usepackage[noadjust]{cite}

\title{Covering $\F_2^n$ with Hamming Balls}
\author{Michael Jaber\thanks{Department of Computer Science, University of Texas at Austin. \href{mailto:mjjaber@cs.utexas.edu}{\texttt{mjjaber@cs.utexas.edu}}. Supported by NSF Grant CCF-2312573, and two Simons Investigator Awards (\#409864, David Zuckerman) and (\#623052, Brent Waters)}
\and 
Vinayak M.\ Kumar\thanks{Department of Computer Science, University of Texas at Austin. \href{mailto:vmkumar@cs.utexas.edu}{\texttt{vmkumar@cs.utexas.edu}}. Supported in part by NSF Grant CCF-2312573, a Simons Investigator Award (\#409864, David Zuckerman), a Jane Street Graduate Research Fellowship, and a UT Austin Dean's Prestigious Fellowship Supplement}
}
  
\date{}
\renewcommand{\backref}[1]{}

\renewcommand*{\backrefalt}[4]{%
    \ifcase #1 %
    No citations.%
    \or
    {#2}%
    \else
    {#2}%
    \fi
}

\newtheorem{lemma}{Lemma}
\newtheorem{proposition}{Proposition}

\newtheorem{question}{Question}

\newtheorem{definition}{Definition}

\usepackage{comment}
\usepackage{nicefrac}
\crefname{prop}{Proposition}{Propositions}
\crefname{ineq}{inequality}{inequalities}
\creflabelformat{ineq}{#2(#1)#3}

\crefname{THM}{Theorem}{Theorems}

\newcommand{\F}{\mathbb{F}}

\newcommand{\supp}{\text{supp}}

\newcommand{\eps}{\varepsilon}

\renewcommand{\subset}{\subseteq}

\begin{document}

\maketitle

\abstract{
Green asked the following question concerning structures in sumsets: Suppose that $\F_2^n$ is partitioned into sets $A_1, \dots, A_K$. Does $A_i+A_i$ contain a coset of codimension $O_K(1)$ for some~$i$? An answer is not known even in the case of $K = 3$. We resolve this question in the affirmative in two special cases:~(1) when $A_1, A_2$ are Hamming balls of radius $r < n/2 - 7$ relative to different bases, and~(2) when $A_1$ is a Hamming ball of sufficiently small constant density. 
}

\section{Introduction}

Understanding the additive structure of sumsets is a major theme in additive combinatorics. A foundational result of Bourgain \cite{bourgain1990arithmetic} showed that if $A \subset [N]$ is dense, then $A+A$ contains a long arithmetic progression. Over finite field vector spaces such as $\F_2^n$, a well-studied question is to determine the largest affine subspace one can always find in $A+A$ when $A$ is a set of density $\alpha$. The techniques of Green \cite{proginsumsetsgreen}\footnote{See \cite{green2003restriction} for a detailed proof.} show that if $\alpha \geq n^{-1/4}$, then $A+A$ contains a coset of dimension $\Omega(\alpha^2 n)$. Sanders \cite{TomSanders2011} removed the restriction on $\alpha$ and showed that $A+A$ contains a coset of dimension $\Omega(\alpha n)$.  

These lower bounds are quite far from the best upper bounds. Indeed, there exist sets $A \subset \F_2^n$ with density $\alpha = \Omega(1)$ where $A+A$ contains no coset of codimension $O(\sqrt{n})$ (see \cite[Theorem 9.4]{Gre05}). The \textit{niveau set} construction, inspired by Ruzsa \cite{ImreZ1991}, is one such example, where $A$ is the set of strings in $\{0,1\}^n$ with weight at most $r = n/2 - c\sqrt{n}$, i.e., a Hamming ball of radius $r$, for an adequately chosen constant $c > 0$. This follows from the central limit theorem, and the fact that every coset of dimension $d$ contains a string of weight at least $d$. In other words, when $A$ is of a fixed constant density, the largest coset one can guarantee to find in $A+A$ has dimension at most $n - O(\sqrt{n})$. On the other hand, the known lower bounds only guarantee that $A+A$ contains a coset of dimension at least $\Omega(n)$, and obtaining tight bounds remains a question of considerable interest.

More generally, one can also ask about the largest coset one can always find in an iterated sumset. A central result due to Sanders \cite{sanders2012bogolyubov} shows that if $A$ has density $\alpha$, then $A+A+A$ contains a coset of codimension $O(\log^{4+o(1)}(1/\alpha))$. The Polynomial Bogolyubov conjecture asks if $kA$ for some fixed constant $k \geq 3$ always contains a coset of codimension $O(\log (1/\alpha))$. This is known to imply the Polynomial Freiman-Ruzsa conjecture, which was resolved in a recent breakthrough \cite{pfrproof}. Both the Polynomial Bogolyubov and Freiman-Ruzsa conjectures have many applications in additive combinatorics and theoretical computer science; see \cite{sanders2012bogolyubov, gs006, pfrproof} and the references therein. 

Towards a better understanding of all of these problems, it is an interesting technical task to identify conditions that guarantee large cosets in $A+A$. One question along these lines is the following, which is listed as Problem 53 in Ben Green's list of 100 Open Problems \cite{GreOp}:

\begin{question}\label{q:partition}
    Suppose $\F_2^n$ is partitioned into sets $A_1, \dots, A_K$. Does $A_i + A_i$ contain a coset of codimension $O_K(1)$ for some $i$?
\end{question}

Even in the setting of $K = 3$, nothing seems to be known. Green notes that a special case of \Cref{q:partition} is to show that $\F_2^n$ cannot be covered by three Hamming balls relative to different bases. Let $B_r(x)$ denote the set of strings that are Hamming distance at most $r$ from $x \in \F_2^n$. For a linear map $L \colon \F_2^n \to \F_2^n$, let $LB_r(x)$ denote the set $\{Ly \mid y \in B_r(x)\}$. Our main result is the following: 

\begin{restatable}{thm}{main}\label{thm:main}
    For $n$ sufficiently large, let $L \colon\F_2^n \to \F_2^n$ be an invertible linear map, and let $A_1 = B_r(x_1)$ and $A_2 = LB_r(x_2)$ for $x_1, x_2 \in \F_2^n$ and $r < n/2 - 7$. Suppose that $\F_2^n = A_1 \cup A_2 \cup S$. Then, $S+S = \F_2^n$. 
\end{restatable}

By applying an appropriate change of basis, it follows easily from \Cref{thm:main} that $\F_2^n$ cannot be covered by three Hamming balls of radius $r < n/2 - 7$ relative to different bases. 

We are also able to obtain a result in the style of \Cref{q:partition} when one of the sets in the covering is a Hamming ball of small constant density.

\begin{restatable}{thm}{twoarb}\label{thm:twoarb}
    For $n$ sufficiently large, there exists an absolute constant $C > 0$ so that if $r < n/2 - C\sqrt{n}$ and
    $$
        \F_2^n = B_r(0) \cup A \cup B,
    $$
    then $A+A$ or $B+B$ contains a coset of codimension 1. 
\end{restatable}

The proof utilizes Spencer's ``six standard deviations suffice'' theorem from discrepancy theory \cite{spencersix}. The codimension 1 in \Cref{thm:twoarb} is optimal, since one can choose $A$ and $B$ to be the set of even and odd weight strings in $B_r(0)^c$, respectively. To the best of our knowledge, prior to \Cref{thm:twoarb}, the best known result in this setting was to argue that $A$ or $B$ is of constant density, then apply the result of \cite{TomSanders2011} to say that $A+A$ or $B+B$ contains a coset of dimension $\Omega(n)$. 

\subsection{Related Work}
\subsubsection{Covering Radius}
Given a subset $S \subset \F_2^n$, the \textit{covering radius} of $S$ is the smallest $r$ so that every point in $\F_2^n$ is Hamming distance at most $r$ from at least one element of $S$. In coding theory, constructing codes with small covering radius is a topic of intense study with numerous applications (see \cite{cohen1997covering}, for example). A central open problem is to determine for a fixed $r$, the size of the smallest set $S \subset \F_2^n$ with covering radius $r$ (see also Problem 40 in \cite{GreOp}). Our work studies a more general problem, where one tries to cover $\F_2^n$ with Hamming balls after having applied an arbitrary linear map to each of the balls. It seems like an interesting direction to see if constructions of covering codes in this model have any new applications in coding theory or computer science.

\subsubsection{Chromatic number of Cayley graphs} 
A special case of \Cref{q:partition} is to show that if $\F_2^n$ is covered by three sets $A, A+x, A+y$ for shifts $x,y \in \F_2^n$, then $A+A$ contains a coset of codimension 1. Indeed, this follows from a result of Payan \cite{PAYAN1992271}, which says that no Cayley graph on $\F_2^n$ can have chromatic number 3. To see why, start with any such covering of $\F_2^n$, and consider the Cayley graph $G = \text{Cay}(\F_2^n, S)$, where $S = (A+A)^c$. Then, $A, A+x, A+y$ are each independent sets in $G$ which cover $\F_2^n$, which means $G$ is 3-colorable. By Payan's theorem, $G$ cannot have chromatic number 3, so $G$ must be bipartite. The result follows since a Cayley graph on $\F_2^n$ is bipartite if and only if there exists some codimension 1 coset $V$ which is disjoint from $S$.

 There are various proofs of Payan's theorem \cite{PAYAN1992271, Cervantes2023ChromaticNO, krebs2026abeliangroups3chromaticcayley}, and it remains open to see if any of the known proofs can be extended to resolve \Cref{q:partition}. Interestingly enough, the proof technique in \Cref{thm:twoarb} can be used to show that if $G$ is a nonbipartite Cayley graph on $\F_2^n$, the induced subgraph on strings of weight $n/2 \pm C\sqrt{n}$ for a large enough constant $C > 0$ is nonbipartite. Overall, progress on \Cref{q:partition} seems intimately related to better understanding proper colorings of Cayley graphs. 

 Kaave Hosseini \cite{Kaaveemail} also observed that an affirmative answer to \Cref{q:partition} would have interesting algorithmic applications for coloring Cayley graphs. In particular, given a Cayley graph $G = \text{Cay}(\F_2^n,S)$ where $\chi(G) = K$, one can find a valid coloring of $G$ using $2^{O_K(1)}$ many colors in time $2^{O_K(n)}$, which for $K = O(1)$ is polynomial in the size of the graph. To see why, let $A_1, \dots, A_K$ be the color classes in a coloring of $G$. If there is some $A_i + A_i$ which contains a coset $V$ of codimension $O_K(1)$, then this means that $S \cap V = \emptyset$. In particular, this implies that $V$ and all of its shifts are independent sets in $G$, which provides a valid coloring using $2^{O_K(1)}$ many colors. In order to find $V$, one can exhaustively search over all possible subspaces in time $2^{O_K(n)}$. 

\subsection{Open Questions}

Our work leaves open two main directions of inquiry. First, \Cref{thm:main} shows that three Hamming balls relative to different bases cannot cover $\F_2^n$. The proof technique does not seem to say much about coverings using four or more balls. When the Hamming balls are relative to the same basis, at radius $r \leq n/2 - 6\sqrt{n}$, any covering requires $\Omega(n)$ balls\footnote{Interestingly, the proof that any such covering requires $\Omega(n)$ balls as well as the proof of \Cref{thm:twoarb} use Spencer's ``six standard deviations suffice'' result \cite{spencersix}.}. When the radius is $r = n/2 - \sqrt{n}/2$, there exist constructions of coverings using $O(n)$ balls. To the best of our knowledge, even if the Hamming balls are relative to different bases, there are no known constructions of coverings using $o(n)$ balls. It remains an interesting question to prove tight upper and lower bounds on the number of balls in such coverings. 

In terms of progress towards \Cref{q:partition}, improving \Cref{thm:twoarb} to work for a Hamming ball of larger radius seems like a modest first step. In particular, any proof which does not heavily rely on the Spencer ``six standard deviations suffice'' theorem seems likely to shed light on \Cref{q:partition}. The proof of \Cref{thm:twoarb} uses Spencer's result to argue that, for certain Cayley graphs on $\F_2^n$, the induced subgraph on strings of weight $n/2 \pm C\sqrt{n}$ has a rich collection of paths. Understanding the properties of induced subgraphs of Cayley graphs seems both challenging and necessary for making progress on this question. One potential direction is the following: are there other structured sets $S$ with the property that if $\F_2^n = A \cup B \cup S$, then $A+A$ or $B+B$ contains a coset of codimension $O(1)$? More generally, it seems natural to ask if techniques from graph theory or discrepancy theory can help say more about \Cref{q:partition}. 

\subsection{Organization}

\Cref{sec:prelim} contains preliminaries and \Cref{sec:main} contains a proof of \Cref{thm:main} assuming three key lemmas. \Cref{sec:pfthm2} contains the proof of \Cref{thm:twoarb}. \Cref{sec:pflem12} contains the proofs of \Cref{lem:upperweight} and \Cref{lem:subspacepart}, and \Cref{sec:pflem3} contains the proof of \Cref{lem:subspacehitsmiddle}. 

\subsection{Acknowledgements}

The authors thank their advisor David Zuckerman for his careful guidance throughout the duration of this work. The authors also thank Geoffrey Mon, Maya Sankar, and Kaave Hosseini for helpful discussions, and Anthony Ostuni for a careful reading that greatly improved the presentation of this manuscript.

\section{Preliminaries}\label{sec:prelim}

For integers $a,b$, we say $a \mid b$ if $a$ divides $b$. For a set $U$, we denote the power set, or set of all subsets of $U$, by $\mathcal{P}(U)$. For a set $A \subset U$, we denote the complement of $A$ in $U$ by $A^c$. 

Let $e_1, \dots, e_n \in \F_2^n$ denote the standard basis vectors. For a point $x \in \F_2^n$, we define the \textit{weight} of $x$ to be the number of indices $i \in [n]$ where $x_i = 1$, which we will denote $\text{wt}(x)$. For two points $x,y \in \F_2^n$, we define their distance $d(x,y)$ to be the number of indices $i \in [n]$ where $x_i \neq y_i$. Note that $d(x,y) = \text{wt}(x + y)$. We denote the Hamming ball of radius $r$ around a point $x \in \F_2^n$ by 
$$
    B_r(x) \coloneqq \{y \in \F_2^n \mid d(x,y) \leq r\}.
$$
For a linear map $L \colon \F_2^n \to \F_2^n$ and a subset $A \subset \F_2^n$, we denote
$$
    L(A) = \{Lx \mid x \in A\}.
$$
Throughout, we will use $LB_r(x)$ as a shorthand for $L(B_r(x))$. One can think of $LB_r(x)$ as the set of points in $\F_2^n$ which have distance $r$ from $Lx$ in the basis given by the columns of the matrix of $L$. We will also use the shorthand $1^n$ to denote the element $(1, \dots, 1) \in \F_2^n$.

For two sets $A, B \subset \F_2^n$, we denote their sumset by $A + B = \{a+b\mid a\in A, b\in B\} \subset \F_2^n$. We will also need the following proposition concerning when $A+A$ contains a codimension 1 subspace, which is implicit in \cite{TomSanders2011}. We reproduce their proof here.  

\begin{proposition}\label{prop:codim1}
    Let $A \subset \F_2^n$. Suppose $A+A$ does not contain any coset of codimension 1. Then, there exists a linear basis $\{v_1, \dots, v_n\}$ and a shift $v \in \F_2^n$ such that $v + \{0, v_1, \dots, v_n\} \subset (A+A)^c$. 
\end{proposition}

\begin{proof}
    Let $S = (A+A)^c$. Suppose that for every $v \in S$, there is no linear basis $\{v_1, \dots, v_n\}$ such that $v + \{0, v_1, \dots, v_n\} \subset S$. In other words, for every $v \in S$, the set $v + S$ contains at most $n-1$ linearly independent vectors. It follows that for every choice of $v \in S$, there exists a coset $H$ of codimension 1 so that $v + S \subset H$. Then, it follows that $(v+H)^c \subset A+A$, which gives the desired result because $(v+H)^c$ is also a coset of codimension 1.
\end{proof}

\section{Proof of \Cref{thm:main}}\label{sec:main}

In this section, we prove our main result, restated below: 

\main*

Our main result will follow from the following three lemmas. Informally, the following lemma states for any choice of $x,y \in \F_2^n$, there exists a linear subspace $V \leq \F_2^n$ such that $x,y \in V$, and there exists an affine shift $v + V$ which only contains strings of high weight in $\F_2^n$.

\begin{restatable}{lemma}{upperweight}\label{lem:upperweight}
    For $n$ sufficiently large, $r \leq n/2-4$, and any $x,y \in \F_2^n$, there exists an affine subspace $v + V \subset B_{n-r}(1^n)$ satisfying $\dim(V) \geq n-r-4$ with the property that $x,y \in V$. 
\end{restatable}

The following lemma is more of a trick that allows us to turn a covering of the form $S \subset A \cup B$ into a covering $S \subset A \cup (x+A)$ as long as $S$ satisfies certain symmetry conditions. 

\begin{restatable}{lemma}{subspacepart}\label{lem:subspacepart}
    Let $A, B \subset \F_2^n$, and suppose $x \in (A+A)^c$ and $y \in (B+B)^c$. Then, if $S \subset A \cup B$ satisfies $S = x+S = y+S$, then  
    $$
        S \subset A \cup (x + A).
    $$
\end{restatable}

The last lemma is a mild extension of a result in \cite{enomoto1987codes}, which says that affine subspaces which avoid strings of weight roughly $n/2$ must have bounded dimension. The original statement in \cite{enomoto1987codes} mentions only linear subspaces, but their proof also works for affine subspaces. We include a proof in \Cref{sec:pflem3} for completeness.

\begin{restatable}[\cite{enomoto1987codes}]{lemma}{subspacehitsmiddle}\label{lem:subspacehitsmiddle}
    For $n$ sufficiently large and $r < n/2-3$, suppose $v \in \F_2^n$ and $V \leq \F_2^n$ is a linear subspace. If $v+V \subset B_r(0) \cup B_r(1^n)$, then $\dim(V) < n/2+3$. 
\end{restatable}

Before we give the proof of \Cref{thm:main}, we will discuss a simpler statement which conveys most of the main ideas. Let $r < n/2 - 7$, and assume for the sake of contradiction that we have a covering of the form 
$$
    \F_2^n = B_r(0) \cup LB_r(x) \cup LB_r(x+1^n),
$$
for some shift $x \in \F_2^n$ and a linear map $L \colon \F_2^n \to \F_2^n$. In other words, we are given a covering of $\F_2^n$ using three Hamming balls, where two of the balls are antipodal and relative to the same basis. This implies that the two antipodal balls must cover the complement of the third. In other words,
$$
    B_r(0)^c \subseteq LB_r(x) \cup LB_r(x+1^n),
$$
where $B_r(0)^c$ is the set of strings in $\F_2^n$ which have weight greater than $r$. From here, obtaining a contradiction is quite simple. \Cref{lem:upperweight} lets us argue that $B_r(0)^c \supseteq B_{n/2 +7}(1^n)$ contains an affine subspace of dimension at least $n/2 + 3$, but \Cref{lem:subspacehitsmiddle} tells us the union of two antipodal Hamming balls for small enough $r$ can only contain affine subspaces of dimension less than $n/2+3$. In the proof of \Cref{thm:main}, the stronger conclusion of \Cref{lem:upperweight} in conjunction with \Cref{lem:subspacepart} lets us reduce from the more general case to this easy case. \\

Now, we will show \Cref{thm:main} assuming the above three lemmas.

\begin{proof}[Proof of \Cref{thm:main}]
    Without loss of generality, assume we have a covering
    $$
        \F_2^n = B_r(0) \cup LB_r(y) \cup S
    $$
    for an invertible linear map $L \colon\F_2^n \to \F_2^n$, a shift $y \in \F_2^n$, and a subset $S \subset \F_2^n$. Since $r < n/2 -7$, we have $B_r(0)^c \supseteq B_{n/2 + 7}(1^n)$, so we have the inclusion
    $$
        B_{n/2+7}(1^n) \subseteq LB_r(y) \cup S.
    $$
    Now, assume towards a contradiction that $S+S \neq \F_2^n$. Then, there exists some $x \in (S+S)^c$, and since $r < n/2$, we have $L1^n \not \in LB_r(y) + LB_r(y) = LB_{2r}(0)$. By \Cref{lem:upperweight}, there exists an affine subspace $v + V \subset B_{n/2+7}(1^n)$ satisfying $\dim(V) \geq n/2+3$ with the property that $L1^n, x \in V$. In particular, we have $V = V+L1^n = V+x$, so $v+V$ is also invariant under the same shifts. The conditions of \Cref{lem:subspacepart} are satisfied, so we have
    $$
        v+V \subset LB_r(y) \cup (L1^n + LB_r(y)).
    $$
    Applying $L^{-1}$ and shifting by $y$ gives
    $$
        w+W \subset B_r(0) \cup B_r(1^n),
    $$
    where $w+W = y + L^{-1}(v+V)$. By \Cref{lem:subspacehitsmiddle}, we have $\dim(V) = \dim(W) < n/2+3$, but this contradicts what we showed earlier, which is that $\dim(V) \geq n/2 +3$.
\end{proof}

\section{Proof of \Cref{thm:twoarb}}\label{sec:pfthm2}

Here, we give a proof of \Cref{thm:twoarb}, restated below: 

\twoarb*

Before we begin, we will need a bit of notation. We define the map $\supp\colon\F_2^n \to \mathcal{P}([n])$ where 
$$
    \supp(x) = \{i : x_i =1 \}.
$$
In other words, $\supp$ views a bit vector $x \in \F_2^n$ as the indicator for a subset of $[n]$, and outputs the corresponding subset. To check understanding, note that $\supp(x+1^n) = \supp(x)^c$, since adding $1^n$ to $x \in \F_2^n$ flips every bit. 

We will also need some notation concerning colorings: if $\chi \colon [n] \to \{-1,1\}$ is a coloring, and $S \subset [n]$ is a subset, then we denote $\chi(S) = \sum_{i \in S} \chi(i)$. A set system $\mathcal{S}$ has \emph{discrepancy} at most $k$ if there exists a coloring $\chi$ so that $\chi(S) \leq k$ for every $S \in \mathcal{S}$. The classic result of Spencer states the following: 

\begin{restatable}[``Six standard deviations suffice'' \cite{spencersix}]{thm}{spencer}\label{thm:spencer}
    Any set system of size $m$ on a universe of size $n$ has discrepancy $O(\sqrt{n \log(m/n)})$. 
\end{restatable}

\paragraph{Warmup to \Cref{thm:twoarb}.} Before we give the proof of \Cref{thm:twoarb}, we will first discuss a simpler case that highlights the main ideas. Let $r < n/2-k$ for some sufficiently large integer $k$, and suppose we have a covering of $\F_2^n$ of the form
$$
    \F_2^n = B_r(0) \cup A \cup (1^n+A),
$$
where $A$ is a set with the property that $1^n + \{0,e_1, \dots, e_n\} \subseteq (A+A)^c$. In particular, this implies that $A+A$ does not contain a coset of codimension 1. Let $T$ denote the set of strings in $\F_2^n$ with weight in the range $n/2 \pm k$. Note that if $t \in T$, we also have $t +1^n \in T$. Since $T \subset B_r(0)^c$, we have
$$
    T \subset A \cup (1^n+A).
$$
From here, we will consider the Cayley graph with vertex set $T$, where two vertices $u,v \in T$ are connected by an edge if $u+v \in \{e_1, \dots, e_n\}$. Note that if $u,u+e_i \in T$, then either $u,u+e_i \in A$ or $u,u+e_i \in (1^n + A)$. If we had $u \in A$ and $u+e_i \in (1^n +A)$, then we would obtain $1^n + e_i \in A+A$, a contradiction. In other words, $A$ and $1^n +A$ are closed under taking paths using edges from $\{e_1, \dots, e_n\}$. If we can argue that there must be some path from $t \in T$ to $t +1^n \in T$, then will obtain a contradiction, since it means $t, t +1^n$ are both in $A$. Indeed, it is not too hard to check that for large enough $n, k$, the Cayley graph on $T$ generated by $\{e_1, \dots, e_n\}$ is connected. 

In general, we might only have the assumption that $1^n + \{0,v_1, \dots, v_n\} \subseteq (A+A)^c$, where $\{v_1, \dots, v_n\}$ is an arbitrary linear basis. Again, we will try to find a path from $t \in T$ to $t+1^n \in T$, where $T$ is the vertex set of the graph, and $\{u,v\}$ is an edge if $u+v \in \{v_1, \dots, v_n\}$ \footnote{In other words, we are trying to find paths in an induced subgraph of the standard Cayley graph generated by $\{v_1, \dots, v_n\}$.}. To do this, we will appeal to \Cref{thm:spencer}. To see why this might be useful, consider the simpler task of trying to find a $t \in T$ with the property that $t+v_1$ is also in $T$. Define the set system 
$$
    \mathcal{S} = \bigg\{[n], \supp(v_1) , \supp(v_1)^c\bigg\}.
$$
If the set system has discrepancy at most $k$, then there exists a coloring $\chi$ with the property that
$$
    \sum_{i \in [n]} \chi(i) \leq k \quad\text{and}\quad \sum_{i \in \supp(v_1)}\chi(i) \leq k \quad\text{and}\quad \sum_{i \not \in \supp(v_1)}\chi(i) \leq k.
$$
Let $t_i = 1$ if $\chi(i) = -1$ and 0 otherwise. Then, the first inequality means that $t \in \F_2^n$ has Hamming weight in the range $n/2 \pm k/2$, and the second and third imply that $t+v_1$ has Hamming weight in the range $n/2 \pm k$. This means $t, t+v_1 \in T$ as desired. In the proof of \Cref{thm:twoarb}, we will use the full strength of \Cref{thm:spencer} to guarantee that many vertices on a desired path are simultaneously in $T$. \\

We now continue with the proof of \Cref{thm:twoarb}.

\begin{proof}[Proof of \Cref{thm:twoarb}]

We start with a covering $\F_2^n = B_r(0) \cup A \cup B$. Assume towards a contradiction that $A+A$ and $B+B$ do not contain a coset of codimension 1. By \Cref{prop:codim1}, there exist linear bases $\{v_1, \dots, v_n\}$ and $\{w_1, \dots, w_n\}$ and shifts $v,w \in \F_2^n$ so that
$$
    v + \{0, v_1, \dots, v_n\} \subset (A+A)^c \quad\text{and}\quad w + \{0, w_1, \dots, w_n\} \subset (B+B)^c. 
$$
Let $U \coloneqq B_r(0)^c$, and consider the set
$$
    T \coloneqq \bigcap_{b \in \{0,1\}^3} (U + b_1 1^n + b_2 v + b_3 w).
$$
Clearly, we have $T \subset U \subset A \cup B$. For now, we will assert that $T$ is nonempty; by the end of the proof we will have found some $t \in T$. Also, we see that $T = 1^n + T = v+T = w+T$. Thus, \Cref{lem:subspacepart} implies $T \subset A \cup (v+A)$. Now, we consider the Cayley graph on $T$ generated by $\{v_1, \dots, v_n\}$, i.e.,~$\text{Cay}(T, \{v_1, \dots, v_n\}) = (V, E)$, where 
$$
    V \coloneqq T \quad\text{and}\quad E = \{\{u,u+v_i\} \mid u,u+v_i \in T, i \in [n]\} .
$$
The claim is that if $\{u,u+v_i\} \in E$, then $u,u+v_i \in A$, or $u,u+v_i \in (v+A)$. Without loss of generality, if $u \in A$ and $u+v_i \in (v+A)$, then we obtain $v+v_i \in A+A$, which is a contradiction. In other words, $A \cap S$ as well as $(v+A) \cap S$ are the union of connected components in $\text{Cay}(T, \{v_1, \dots, v_n\})$. Thus, if we can find a path from $t \in T$ to $t+v \in T$, then we will either have $t,t+v \in A$ or $t,t+v \in (v+A)$, both of which imply $v \in A+A$, a contradiction. Since $\{v_1, \dots, v_n\}$ forms a linear basis, we can write $v = v_{i_1} + v_{i_2} + \dots + v_{i_k}$ for some subset $\{i_1, \dots, i_k\} \subset [n]$. Thus, we will try to derive that the vertices on the path 
$$
    t, t+v_{i_1}, t+v_{i_1} + v_{i_2},  \dots, t + \sum_{j=1}^\ell v_{i_j}, \dots, t + v
$$
are all contained in $T$. 

We will use \Cref{thm:spencer} to show this. Define the set system $\mathcal{S}$ by 
$$
    \mathcal{S} \coloneqq \bigcup_{\ell=0}^k \bigcup_{b\in \{0,1\}^3} \left\{\supp\left(b_1 1^n + b_2 v + b_3 w + \sum_{j=1}^\ell v_{i_j} \right) \right\},
$$
where we define $\sum_{j=1}^0 v_{i_j} = 0$. We record three key facts about $\mathcal{S}$:

\begin{enumerate}
    \item By setting $b = 0^3$, for every choice of $\ell = 0\dots k$, we have $\supp(\sum_{j=1}^\ell v_{i_j}) \in \mathcal{S}$.
    \item By fixing $\ell, b_2, b_3$, we see that if $S \in \mathcal{S}$, we also have $S^c \in \mathcal{S}$. This is because if $\mathcal{S}$ contains $\supp(u)$ for some $u \in \F_2^n$, we also have $\supp(1^n + u) = \supp(u)^c \in \mathcal{S}$.
    \item Similarly, if $\supp(u) \in \mathcal{S}$, we also have $\supp(u+v), \supp(u+w) \in \mathcal{S}$. 
\end{enumerate}  

There are at most $8(k+1) \leq 8(n+1)$ many sets in $\mathcal{S}$, so an application of \Cref{thm:spencer} implies that there exists a coloring $\chi \colon [n] \to \{-1,1\}$ so that $|\chi(S)| \leq C\sqrt{n}$ for every $S \in \mathcal{S}$ where $C$ is an absolute constant. We will now show that the desired $t \in T$ is given by $t_i = (1-\chi(i))/2$, i.e.,~ $t_i = 1$ if $\chi(i) = -1$ and $t_i = 0$ if $\chi(i) = 1$.

Consider an arbitrary $S \in \mathcal{S}$ of the form $S = \supp(u)$ for some $u \in \F_2^n$. Fact (2) implies that $S^c \in \mathcal{S}$. Thus, we have $\chi(S), \chi(S^c) \leq C\sqrt{n}$, which means
$$
    \left|\sum_{i \not\in S} \chi(i) - \sum_{j \in S} \chi(j) \right| \leq 2C\sqrt{n}.
$$
If we rewrite the above sum, we have
$$
    \sum_{i \not\in S} \chi(i) - \sum_{j \in S} \chi(j) = \sum_{i \in [n]} \chi(i) (-1)^{\mathbf{1}\{i \in S\}} = \sum_{i \in [n]}(-1)^{t_i} (-1)^{\mathbf{1}\{i \in S\}} = \sum_{i \in [n]} (-1)^{t_i \oplus u_i},
$$
which is equivalent to saying that $t + u$ has Hamming weight within $C\sqrt{n}$ of $n/2$, i.e.,~$t + u \in U \cap (1^n + U)$. By applying Facts (2) and (3), we also know that
$$
    \supp(b_1 1^n + b_2 v + b_3w + u)\in \mathcal{S} \text{ for every }b \in \{0,1\}^3,
$$
which by the same argument implies that $t + b_2 v + b_3w + u\in U \cap (1^n + U)$ for every $b_2, b_3 \in \{0,1\}$. Recalling the definition of $T$, we see that this is equivalent to $t + u \in T$. Since the choice of $S \in \mathcal{S}$ was arbitrary, it follows from Fact (1) that every vertex $t+ \sum_{j=1}^\ell v_{i_j}$ for $\ell = 0 \dots k$ on our desired path is in $T$, so we obtain the desired contradiction. \qedhere
\end{proof}

\section{Proofs of \Cref{lem:upperweight} and \Cref{lem:subspacepart}}\label{sec:pflem12}

Here  we prove \Cref{lem:upperweight}, which says that for any $x,y \in \F_2^n$, there exists a large linear subspace $V \leq \F_2^n$ so that $x,y \in \F_2^n$, and there exists an affine shift $v + V$ which only contains strings of high weight.

\upperweight*

\begin{proof}
    Let $s = |\{i \mid x_i = y_i = 1\}|$, $t = |\{i \mid x_i = 1, y_i = 0\}|$, $u = |\{i \mid x_i = 0, y_i = 1\}|$, and $v = n - (s+t+u)$. After permuting coordinates, we can assume that $x = 1^{s+t}0^{u+v}$ and $y = 1^s0^t1^{u}0^v$. Since $B_{n-r}(1^n)$ is symmetric under permuting coordinates, it suffices to find the desired subspace $V$ with these choices of $x,y \in \F_2^n$. Now, let $k = 2\lceil r+3\rceil$, and set
    $$
        W = \text{span}\{e_1+e_2, e_3 + e_4, \dots, e_{k-1} +e_k\} + \text{span}\{ \{e_i\}_{i > k}\}
    $$
    with $v = \sum_{i=1}^k e_{2i}$. At this point, it is worth noting that every element in $v+W$ has the property that exactly one of the coordinates at indices $2i-1$ and $2i$ are set to 1 for $1 \leq i \leq k/2$. We then set $V = W + \text{span}\{e_s, e_{s+t},e_{s+t+u}\}$\footnote{If $s=0$, we can trivially set $e_0 = 0^n$.}. This guarantees we have $x,y \in V$. Also, every element in $v + W$ has weight at least $k/2$, and thus every vector in $v + V$ has weight at least $k/2 - 3 \geq r$. The last point to check is the dimension of $V$, which is at least
    \begin{align*}
        \dim(V) \geq \dim(W) &= k/2 + (n-k) \\
        &= n - k/2 \\
        &\geq n-r-4. \qedhere
    \end{align*} 
\end{proof}

The following lemma shows that if we have a covering $S \subset A \cup B$ where $S$ satisfies certain symmetry conditions, then we also have the covering $S \subset A \cup (x+A)$. Coverings of this form end up being simpler for us to work with in both \Cref{thm:main} and \Cref{thm:twoarb}. 

\subspacepart*

\begin{proof}
 It suffices to prove the statement when $S = A \cup B$. First, note that if $z \in A$, then $x+z \in B$. This follows because $x+z \in S$ since $S = x+S$, but $x +z \not \in A$ since $x \in (A+A)^c$. Similarly, if $z \in B$, then $y+z \in A$. Thus, we have $x+A \subset B$, and $y+B \subset A$. These together imply that $|A| = |B|$, and in particular, that $x+A = B$ as desired.  
\end{proof}

\section{Proof of \Cref{lem:subspacehitsmiddle}}\label{sec:pflem3}

For the sake of completeness, we include a proof \Cref{lem:subspacehitsmiddle}, which was originally stated and proven in \cite{enomoto1987codes}.

\subspacehitsmiddle*

We note that \Cref{lem:subspacehitsmiddle} is stated in \cite{enomoto1987codes} in different language for linear subspaces $V$, but their result extends to affine subspaces $v+V$ with little effort. There, the authors were concerned with understanding the dimension of linear codes which avoid codewords of certain weights. For the sake of completeness, we provide their proof in this section.

An important notion will be the definition of a matrix in \textit{binormal form}. 

\begin{definition}[Binormal form]
    Let $M \in \F_2^{k \times n}$ be a matrix with $2k \leq n$. Let $c_1, \dots, c_n \in \F_2^k$ be the columns of $M$. We say that $M$ is in binormal form if for $i = 1, \dots, k$, we have
    $$
        c_{2i-1} + c_{2i} = e_i \in \F_2^k.
    $$
\end{definition}
Stated differently, the $(2i-1)$-th and $(2i)$-th column differ only in the $i$-th coordinate. Clearly, a matrix in binormal form has rank exactly $k$, since the vectors $e_1, \dots, e_k$ are in the span of the columns. As we will soon see, the important property of matrices in binormal form is that every element in $\F_2^k$ can be written using the sum of exactly $k$ columns. For a string $\eps = (\eps_1, \dots, \eps_k) \in \{0,1\}^k$, we define 
$$
    b(\eps) = \sum_{1 \leq i \leq k} b_{2i-\eps_i} \in \F_2^k.
$$
In other words, the sequence $\eps$ tells you which columns to take, where $\eps_i=1$ indicates to take $b_{2i-1}$ in the sum, and $\eps=0$ indicates to take $b_{2i}$. For example, a sequence of $\eps = 1^k$ would set $b(\eps)$ to be the sum of the odd indexed columns from $1$ to $2k$. The following proposition demonstrates the key property we will need of matrices in binormal form. 

\begin{proposition}\label{prop:binormalweightk}
    If $M \in \F_2^{k \times n}$ is in binormal form, and $v \in \F_2^k$, then there exists a unique choice for $\eps \in \{0,1\}^k$ such that $b(\eps) = v$.  
\end{proposition}
\begin{proof}
    Let $\eps, \delta \in \{0,1\}^k$ be distinct sequences, and consider $b(\eps) + b(\delta)$. We have
    $$
        b(\eps) + b(\delta) = \sum_{1 \leq i \leq k} b_{2i-\eps_i} + b_{2i - \delta_i} = \sum_{1 \leq i \leq k}  \mathbf{1}\{\eps_i \neq \delta_i\} \cdot e_i.
    $$
    Since $\eps \neq \delta$, there exists an index $i$ where $\eps_i \neq \delta_i$, which implies the above vector is nonzero since the $e_i$ are linearly independent. Thus, the map $b \colon \{0,1\}^k \to \F_2^k$ is a bijection. 
\end{proof}

One helpful consequence of \Cref{prop:binormalweightk} is the following: if $M \in \F_2^{k \times n}$ is in binormal form, then for every $y \in \F_2^k$, there exists some $x \in \F_2^n$ such that $\text{wt}(x) = k$ and $Mx = y$. This is because if there exists some $\eps \in \{0,1\}^k$ so that $b(\eps) = y$, then we know that there are exactly $k$ columns of $M$ whose sum is $y$. We will use this idea to show that every large enough affine subspace contains an element of weight roughly $n/2$, which suffices to prove \Cref{lem:subspacehitsmiddle}.

In order to put a matrix into binormal form, we will need the following lemma. 

\begin{lemma}\label{lem:algoforbinorm}
    Let $M \in \F_2^{k \times n}$ be a full-rank matrix with $n$ odd, $2k < n$, and every row of $M$ orthogonal to $1^n$ (i.e.,~$M 1^n = 0$). Then, $M$ can be brought to binormal form by row operations and permutations of the columns.
\end{lemma}

\begin{proof}
    We apply induction on $k$. In the base case of $k=1$, the single row must have a $0$ and $1$ (the row being $0^n$ contradicts full rank, and the row being $1^n$ contradicts being orthogonal to $1^n$ since $n$ is odd). Thus, we can permute the columns so that the first two bits are different, which means the matrix is in binormal form.

    Now consider a $k\times n$ matrix with full rank and all rows orthogonal to $1^n$. The $(k-1)\times n$ submatrix obtained by removing the last row also satisfies these properties, so we can apply the inductive hypothesis on this submatrix to bring it into binormal form. 

    To bring the entire matrix into binomial form, it suffices to ensure that in the last row,  $M_{k,2i-1} = M_{k,2i}$ for $i < k$, and $M_{k,2k-1} = M_{k,2k}$. For $1\le i < k$, if $M_{k,2i-1} \neq M_{k,2i}$, we can add the $i$-th row to the last row in order to make the entries equal. Notice that since the upper $(k-1) \times n$ submatrix is in binormal form, any two entries $M_{k,2j-1}, M_{k,2j}$ for $1 \leq j \leq k$ that are equal will remain equal after adding the $i$-th row.
    
    All that remains is to ensure that $M_{k,2k-1} \neq M_{k,2k}$. For simplicity, we will denote the last row vector by the vector $r \in \F_2^n$. Notice that if there existed a 0 and 1 entry past the $(2k-2)$-th entry of $r$, then we can use column permutations to bring them to the $(2k-1)$-th and $2k$-th entry. At this point, we might have caused some violations to binormality in the $(2k-1)$-th or $2k$-th columns. By summing $r$ with any previous row whose $(2k-1)$-th and $2k$-th entry are distinct, we make the overall matrix binormal. This corrects any violations to binormality in the $(2k-1)$-th and $2k$-th columns, and it will not create new violations in any other column, since $M_{k,2i-1} = M_{k,2i}$ for $1 \leq i < k$. 
    
    We now have to deal with the case $r_i = 0$ for all $i > 2k-2$ (the all $1$s case is not possible due to orthogonality with $1^n$). Because $M$ has full rank, there must exist some $i \leq k-1$ such that $(r_{2i-1},r_{2i}) = (1,1)$ (otherwise $r = 0$). We can now swap columns $2i-1$ and $2i$ with columns $2k-3$ and $2k-2$, respectively, and then swap rows $i$ and $k-1$ (via row operations). As an example, the given submatrix $(M_{i,j})_{1 \leq i \leq k, 1 \leq j \leq 2k}$ will undergo the following operations, starting from the assumption that $(r_1, r_2) = (1,1)$: 
    $$
    \begin{pmatrix}
        10 & \dots &11& 00 \\
        \vdots & & \vdots & \vdots \\
        00 & \dots & 01 & 11 \\
        11 & \dots &00 & 00
    \end{pmatrix} \xrightarrow[\text{swap cols }2,2k-2]{\text{swap cols 1,}2k-3}
    \begin{pmatrix}
        11 & \dots & 10 & 00 \\
        \vdots & & \vdots & \vdots \\
        01 & \dots & 00 & 11 \\
        00 & \dots & 11 & 00
    \end{pmatrix} \xrightarrow{\text{swap rows }1, k-1}
    \begin{pmatrix}
        01 & \dots &00& 11 \\
        \vdots & & \vdots & \vdots \\
        11 & \dots & 10 & 00 \\
        00 & \dots & 11 & 00
    \end{pmatrix}$$
    This operation preserves binormality of the upper $(k-1)\times n$ submatrix while also setting $r_{2k-3}=r_{2k-2} = 1$. In particular, we also have $M_{k-1, 2k-3} \neq M_{k-1, 2k-2}$. We can assume that $M_{k-1,2k-3} = 0$ and $M_{k-1,2k-2}= 1$ by swapping columns. At this point, the submatrix $(M_{i,j})_{k-1 \leq i \leq k, 2k-3 \leq j \leq 2k+1}$ is of the form $\begin{pmatrix}
        0 & 1 & \ast & \ast & \ast \\
        1 & 1 & 0 & 0 & 0
    \end{pmatrix}$. Additionally, $r_i = 0$ for all $i \geq 2k-1$. 
    
    There are now two cases: if $M_{k-1,j} = 1$ for all $j \geq 2k-1$, then after adding the bottom two rows, we will have the submatrix $(M_{i,j})_{k-1\le i\le k, 2k-3\le j\le 2k+1} = \begin{pmatrix}0&1&1&1&1 \\ 1&0&1&1&1\end{pmatrix}$. Swapping the columns corresponding to the 2nd and 3rd column in this submatrix now puts the bottom row, as well as the $(2k-2)\times n$ left-submatrix, in binormal form. We now repeat the trick of adding the last row to any row $i$ above it with $M_{i,k-1}\neq M_{i,k}$. This gives us the desired binormal form.
    
    If $M_{k-1,j} = 0$ for some $j \geq 2k-1$, since the rows are orthogonal to $1^n$ and $n$ is odd, there must exist some $j'\geq 2k-1$ such that $M_{k-1,j'} = 1$. Furthermore, there must exist another $j'' \geq 2k-1$ with $j \neq j''$ such that $M_{k-1,j''} = 0$. By permuting columns, we may assume $j = 2k-1$, $j' =2k$, and $j'' = 2k+1$, and so we have the following submatrix $(M_{i,j})_{k-1\le i\le k, 2k-3\le j\le 2k+1} = \begin{pmatrix}
        0 & 1 & 0 & 1 & 0 \\ 1 & 1 & 0 & 0 & 0
    \end{pmatrix}$. After permuting columns, we can get to the following induced submatrix $\begin{pmatrix}
        0 & 1 & 0 & 0 & 1 \\ 0 & 0 & 0 & 1 & 1
    \end{pmatrix}$. This submatrix is in binormal form. At this point, the only violations to binormality can be found in the pairs of elements in the $(2k-3)$-th and $(2k-2)$-th columns, or the $(2k-1)$-th and $2k$-th columns. We again repeat the trick of adding the $(k-1)$-th or $k$-th row to the appropriate row where $M_{i,2k-3} \neq M_{i,2k-2}$ or $M_{i,2k-1} \neq M_{i,2k}$, respectively, which puts the entire matrix into binormal form. 
\end{proof}

From here, we can prove \Cref{lem:subspacehitsmiddle}. 

\begin{proof}[Proof of \Cref{lem:subspacehitsmiddle}]
    We first treat the case where $4 \mid n$, showing that every affine subspace
    $v+V \subset \mathbb{F}_2^n$ with $\dim(V) \ge n/2$ contains a vector of
    Hamming weight $n/2$ or $n/2-1$. The case of general $n$ reduces to this as follows.
    Write $n = 4m + k$ with $0 \le k \le 3$. By ignoring the last $k$ coordinates,
    the affine subspace $v+V \subset \mathbb{F}_2^n$ with $\dim(V) \ge n/2 + 3$
    projects to an affine subspace $v' + V' \subset \mathbb{F}_2^{n-k}$ satisfying
    $\dim(V') \ge n/2 + 3 - k \ge (n-k)/2$. 
    
    Since $4 \mid (n-k)$, the base case implies that $v' + V'$ contains a vector of
    weight $(n-k)/2$ or $(n-k)/2 - 1$. Lifting this vector back to
    $\mathbb{F}_2^n$, we conclude that $v+V$ contains a vector whose weight lies in
    the interval
    $[\tfrac{n-k}{2}-1,\; \tfrac{n-k}{2}+k]$.
    When $r \le n/2 - 3$, this contradicts the assumption that
    $v+V \subset B_r(0) \cup B_r(1^n)$, completing the reduction.
        
    Assume towards a contradiction that $4 \mid n$, and there exists an affine subspace $v+V$ with $\dim(V) = n/2$ that does not contain an element of weight $n/2$ or $n/2-1$. Without loss of generality, we may assume that $1^n \in V$, since otherwise, we may consider $V' = \text{span}\{1^n, V\}$, which has dimension $n/2 + 1$, which does not contain a word of weight $n/2$ or $n/2+1$. Then, we can take $V''$ to be the set of elements of $V'$ with even weight, which will decrease the dimension by at most 1, and will give an affine subspace of dimension at least $n/2$ which contains no element of weight $n/2$ or $n/2-1$.
    
    Now, consider $V^{\perp}$, and note that $\dim(V^\perp) = n/2$. Let $M_0 \in \F_2^{n/2 \times n}$ be the matrix whose rows are a basis for $V^{\perp}$. Since $\ker(M_0) = V$, the preimages of elements in $\F_2^{n/2}$ under $M_0$ are shifts of $V$, i.e.,~we have
    $$
         v+V = \{x \mid M_0x = M_0v\}.
    $$
    Now, append a column of the all zero vector $0^{n/2} \in \F_2^{n/2}$ to $M_0$ to create $M$. Note that $M_0x = y$ if and only if $M (x\circ b) = y$, where $x \circ b$ denotes the vector obtained by concatening a bit $b \in \{0,1\}$ to $x$. Since $1^n \in V$, all of the rows of $M$ are orthogonal to $1^n$, which means that we can apply \Cref{lem:algoforbinorm} to put $M$ into binormal form, which we will call $M'$. Note that since $M'$ is obtained from $M$ by row operations and permuting columns, there exists an invertible matrix $R$ and a permutation matrix $P$ such that $M' = RMP$. We want to use \Cref{prop:binormalweightk} to find a vector $x$ of weight $n/2$ or $n/2-1$  so that $M_0 x = M_0 v$. It suffices to find a vector $y \in \F_2^{n+1}$ of weight $n/2$ such that $M'y = RM_0v$. To see why, recall that $R$ is invertible, and that
    $$
        M'y = RM_0v \implies RMPy = RM_0v \implies MPy = M_0v.
    $$
    $P$ is a permutation matrix, so $Py$ is a vector of weight $n/2$ with the property that $MPy = M_0v$. Since the last column of $M$ is all zeros, cutting off the last coordinate of $y$ gives a vector $y'$ of weight $n/2$ or $n/2-1$ with $M_0y' = M_0v$, which means $y' \in v+V$. By \Cref{prop:binormalweightk}, our desired choice of $y$ exists, since there exists a choice of $\eps \in \{0,1\}^{n/2}$ such that $b(\eps) = RM_0 v$.
\end{proof}

\bibliographystyle{alpha} \bibliography{main.bib}
\end{document}